\documentclass[12pt,reqno,fleqn]{amsart}  
\usepackage{tikz}
\usepackage{amsmath,amstext,amsthm,amssymb,amsxtra}
\usepackage[top=1.5in, bottom=1.5in, left=1.25in, right=1.25in]	{geometry}
\usepackage{lmodern}

\usepackage{graphics}
\graphicspath{{./images/}}

\usepackage{mathtools}
\mathtoolsset{showonlyrefs,showmanualtags}

\usepackage{hyperref} 
\hypersetup{
	colorlinks=true,       
	linkcolor=blue,          
	citecolor=magenta,        
	filecolor=magenta,      
	urlcolor=cyan           
}

\usepackage[msc-links]{amsrefs}

\newtheorem{theorem}{Theorem}[section]

\newtheorem{lemma}{Lemma}[section]
\newtheorem{corollary}{Corollary}[section]
\newtheorem{OldTheorem}{Theorem}
\theoremstyle{definition}
\newtheorem{definition}{Definition}[section]
\theoremstyle{definition}

\theoremstyle{remark}
\newtheorem{remark}{Remark}[section]

\def\ZF{\ensuremath{\mathcal F}}

\def\OSC{{\rm OSC}}

\def\MM^d{\ensuremath{\mathfrak M}}
\def\MM{\ensuremath{\mathcal M}}

\def\ZA{\ensuremath{\mathcal A}}

\def\ZZ{\ensuremath{\mathbb Z}}
\def\ZI{\ensuremath{\textbf 1}}

\def\ZN{\ensuremath{\mathbb N}}

\def\ZP{\ensuremath{\mathcal P}}

\def\zT{\ensuremath{\mathcal T}}
\def\ZR{\ensuremath{\mathbb R}}

\def\ZT{\ensuremath{\mathbb T}}

\def\ZH{\ensuremath{\mathcal H}}

\numberwithin{equation}{section}
\def\md#1#2\emd{\ifx0#1
	\begin{equation*} #2 \end{equation*}\fi  
	\ifx1#1\begin{equation}#2\end{equation}\fi   
	\ifx2#1\begin{align*}#2\end{align*}\fi   
	\ifx3#1\begin{align}#2\end{align}\fi    
	\ifx4#1\begin{gather*}#2\end{gather*}\fi  
	\ifx5#1\begin{gather}#2\end{gather}\fi   
	\ifx6#1\begin{multline*}#2\end{multline*}\fi  
	\ifx7#1\begin{multline}#2\end{multline}\fi  
	\ifx8#1\begin{multline*}\begin{split}#2\end{split}\end{multline*}\fi
	\ifx9#1\begin{multline}\begin{split}#2\end{split}\end{multline}\fi
}
\newcommand {\e }[1]{\eqref{#1}}
\newcommand {\lem }[1]{Lemma \ref{#1}}
\newcommand {\rem }[1]{Remark \ref{#1}}
\newcommand {\cor }[1]{Corollary \ref{#1}}

\newcommand {\trm }[1]{Theorem \ref{#1}}

\newcommand {\sect }[1]{Section \ref{#1}}

\title[] {On wavelet polynomials and Weyl multipliers}
\author{Anna Kamont}
\address{Institute of Mathematics, Polish Academy of Sciences, ul. Abrahama 18, 80–825 Sopot, Poland}
\email{anna.kamont@impan.pl}

\author{Grigori A. Karagulyan}
\address{Faculty of Mathematics and Mechanics, Yerevan State
	University, Alex Manoogian, 1, 0025, Yerevan, Armenia} 
\email{g.karagulyan@ysu.am}

\thanks{The second author is supported by the Science Committee of Armenia (the grant number will be provided later)}

\subjclass[2010]{42C05, 42C10, 42C20, 42C40}
\keywords{wavelet type systems, Franklin system, non-overlapping polynomials, Weyl multiplier, Menshov-Rademacher theorem}
\begin{document}
	\begin{abstract}
		For the wavelet type orthonormal systems $\phi_n$, we establish a new bound
			\begin{equation}
			\left\|\max_{1\le m\le n}\left|\sum_{j\in G_m}\langle f,\phi_j\rangle \phi_j\right|\right\|_p\lesssim \sqrt{\log (n+1)}\cdot \|f\|_p,\quad 1<p<\infty,
		\end{equation}
	where $G_m\subset \ZN$ are arbitrary sets of indexes.	Using this estimate, we prove that $\log n$ is an almost everywhere convergence Weyl multiplier for any orthonormal system of non-overlapping wavelet polynomials. It will also be remarked that  $\log n$ is the optimal sequence in this context.
	\end{abstract}

	\maketitle  
	\section{Introduction}
	
	\subsection{An historical overview}

	Recall some definitions well-known in the theory of orthogonal series (see \cite{KaSa}).
	\begin{definition}
		Let $\Phi=\{\phi_n:\, n=1,2,\ldots\}\subset L^2(0,1)$ be an orthonormal system. A sequence of positive numbers $\omega(n)\nearrow\infty$ is said to be an a.e. convergence Weyl multiplier (C-multiplier) if every series 
		\begin{equation}\label{a1}
			\sum_{n=1}^\infty a_n\phi_n(x),
		\end{equation}
		with coefficients satisfying the condition $\sum_{n=1}^\infty a_n^2w(n)<\infty$ is a.e. convergent. If such series converges a.e. after any rearrangement of the terms, then we say $w(n)$ is an a.e unconditional convergence Weyl multiplier (UC-multiplier) for $\Phi$.
	\end{definition}
	The Menshov-Rademacher classical theorem (\cite{Men},  \cite{Rad}, see also \cite{KaSa} or \cite{KaSt}) states that the sequence $\log^2 n$ is a C-multiplier for any orthonormal system. The sharpness of $\log^2 n$ in this theorem was proved by Menshov in the same paper \cite{Men}. That is any sequence $w(n)=o(\log^2n)$ fails to be C-multiplier for some  orthonormal system. The following inequality is the key ingredient in the proof of the Menshov-Rademacher theorem.
	\begin{OldTheorem}[Menshov-Rademacher, \cite{Men}, \cite{Rad}]\label{MR}
		If $\{\phi_k:\, \,k=1,2,\ldots,n\}\subset L^2(0,1)$ is an orthogonal system, then
		\begin{equation}\label{1}
			\left\|\max_{1\le m \le n}\left|\sum_{k=1}^m\phi_k\right|\,\right\|_2\le c\cdot\log (n+1) \left\|\sum_{k=1}^n\phi_k\right\|_2,
		\end{equation}
		where $c>0$ is an absolute constant.
	\end{OldTheorem}
	Similarly, the counterexample of Menshov is based on the following results. 
	\begin{OldTheorem}[Menshov, \cite{Men}]\label{M}
		For any natural number $n\in\ZN$ there exists an orthonormal system $\{\phi_k:\, k=1,2,\ldots,n\}\subset L^2(0,1)$, such that
		\begin{equation}
			\left|\left\{x\in(0,1):\,\max_{1\le m \le n}\left|\frac{1}{\sqrt{n}}\sum_{k=1}^m\phi_k(x)\right|\ge c\log (n+1)\right\}\right|\gtrsim 1,
		\end{equation}
		for an absolute constant $c>0$.
	\end{OldTheorem}
\trm{M} shows that the factor $\log (n+1)$ is optimal in inequality \e{1} when the class of all orthogonal systems is considered. However, if we narrow the class of 
orthogonal systems, it is of interest to find the best constant in inequality \e{1} for the class under consideration. 

For the further discussion, we need to introduce some notation.
The relation $a\lesssim b$ ($a\gtrsim b$) will stand for the inequality $a\le c\cdot b$ ($a\ge c\cdot b$), where $c>0$ is an absolute constant. Given two sequences of positive numbers $a_n,b_n>0$, we write $a_n\sim b_n$ if we have $c_1\cdot a_n\le b_n\le c_2\cdot a_n$, $n=1,2,\ldots$ for some constants $c_1, c_2>0$.  Throughout the paper, the base of $\log$ is equal $2$.
	
	Let $\Phi=\{\phi_k(x),\,k=1,2,\ldots\}\subset L^\infty(0,1) $ be an infinite orthonormal system.  Given $g\in L^1(0,1)$, consider the Fourier coefficients
	\begin{equation*}
		a_n=\langle g,\phi_n\rangle=\int_0^1g\phi_n.
	\end{equation*} 
	Let $f$ be a $\Phi$-polynomial (i.e. a finite linear combination of functions of $\Phi$). Then we say $f$ satisfies the relation $f\prec g$ (with respect to the system $\Phi$), if $f=\sum_{n=1}^m \lambda_na_n\phi_n$, where $|\lambda_n|\le 1$.
	Given orthonormal system $\Phi$ and $1<p<\infty$ we consider the numerical sequence
	\begin{equation}\label{u33}
		\ZA_n^p(\Phi)=\sup_{\|g\|_p\le 1}\,\sup_{g_k\prec g,\, k=1,\ldots,n}\left\|\max_{1\le k\le n}|g_k(x)|\right\|_p,
	\end{equation}
	where the second $\sup$ is taken over all sequences of polynomials $g_k$, $k=1,2,\ldots,n,$ satisfying $g_k\prec g$.
	One can consider in \e{u33} only the monotonic sequences of polynomials
	\begin{equation*}
		g_k=\sum_{j\in G_k}\langle g,\phi_k\rangle\phi_j\prec g,\quad k=1,2,\ldots, n,
	\end{equation*}
	where $G_1\subset G_2\subset \ldots \subset G_n\subset \ZN$, $\#G_n<\infty$. Then we will get another sequence $\ZA^p_{n,\text{mon}}(\Phi)$. If we additionally suppose that
	each $G_{k+1}\setminus G_k$ consists of a single integer, then we will have the sequence $\ZA^p_{n,\text{sng}}(\Phi)$. Clearly we have
	\begin{equation}\label{1-2}
		\ZA^p_{n,\text{sng}}(\Phi)\le\ZA^p_{n,\text{mon}}(\Phi)\le\ZA^p_n(\Phi).
	\end{equation}
	The case of $p=2$ is special for \e{u33}. \trm{MR} implies $\ZA^2_{n,\text{mon}}(\Phi)\lesssim \log (n+1)$  for every orthonormal system $\Phi$.
	On the other hand, applying \trm{M}, one can also construct an infinite orthonormal system with the lower bound $\ZA^2_{n,\text{sng}}(\Phi)\gtrsim \log (n+1)$, $n=1,2,\ldots $. Thus we conclude that for the general orthonormal systems the logarithmic upper bound of $\ZA^2_{n,\text{mon}}(\Phi)$  is optimal. As it was remarked in \cite{Kar1} from some results of Nikishin-Ulyanov \cite{NiUl} and Olevskii \cite{Ole}  it follows that $\ZA^2_{n,\text{mon}}(\Phi)\gtrsim \sqrt{\log (n+1)}$ for any complete orthonormal system $\Phi$. 
	
	The recent papers of one of the authors \cite{Kar1, Kar2, Kar3} highlight the relation of sequences \e{1-2} in the study of almost everywhere convergence of special orthogonal series. It was proved in \cite{Kar1} that
	\begin{OldTheorem}\label{OT1}
		If $\Phi$ is a martingale difference, then  $\ZA^2_{n,\text{mon}}(\Phi)\lesssim \sqrt{\log (n+1)}$.
	\end{OldTheorem}
	\begin{OldTheorem}\label{OT2}
		For any generalized Haar system $\ZH$ we have the relation 
		\begin{equation}\label{x23}
			\ZA^2_{n,\text{sng}}(\ZH)\sim \ZA^2_{n,\text{mon}}(\ZH) \sim \sqrt{\log (n+1)}.
		\end{equation}
	\end{OldTheorem}
	The paper \cite{Kar1} also provides corollaries of these results like those that will be considered below.
	In the case of trigonometric system \cite{Kar2, Kar3} it was proved the following.
	\begin{OldTheorem}
		If $\zT$ is the trigonometric system, then we have 
		\begin{equation}\label{x6}
			\ZA^2_{n,\text{sng}}(\zT)\sim \ZA^2_{n,\text{mon}}(\zT) \sim \log (n+1).
		\end{equation}
	\end{OldTheorem}
	Note that the upper bound $\ZA^2_{n,\text{mon}}(\zT) \lesssim \log (n+1)$ in \e{x6} follows from the Menshov-Rademacher theorem. So the novelity here is the estimate $\ZA^2_{n,\text{sng}}(\zT)\gtrsim \log (n+1)$, which shows that the trigonometric system has no better estimate of the sequences $\ZA^2_n$ than the general orthonormal systems have.

\subsection{The statement of the new results}	
	In this paper we will consider orthonormal systems of wavelet type $\Phi=\{\phi_k(x):\,k=1,2,\ldots\}$ defined on the interval $[0,1]$. For the precise definition of such systems we use the function
	\begin{equation}\label{u51}
		\xi(x)=\frac{1}{(1+|x|)^{1+\delta}},\quad 0<\delta<1,
	\end{equation}
	as well as the notations 
	\begin{align}
		&t_1=\frac{1}{2},\quad t_k=\frac{2j-1}{2^{n+1}},\quad k\ge 2,\\
		&\text{where }	k=2^n+j,\, 1\le j\le 2^n,\,n=0,1,2,\ldots.\label{u52}
	\end{align}
	Hence we suppose 
	\begin{align}
		&\int_0^1\phi_k(t)dt=0,\quad k\ge 1,\label{u65}\\
		&|\phi_k(x)|\le c2^{n/2}\cdot \xi\big(2^n(x-t_k)\big),\label{u43}\\
		&|\phi_k(t)-\phi_k(t')|\le c2^{n/2}(2^n|t-t'|)^\alpha \cdot \xi\big(2^n(t-t_k)\big)\text { if }|t-t'|\le 2^{-n}, \label{u44}
	\end{align}
	where $0<\alpha\le 1$.  Recall that such wavelet type systems in an other context were considered in \cite{Mul}. Note that the Fourier series of any function $f\in L^p$, $1<p<\infty$, in a wavelet type system 
	converges in $L^p$ unconditionally (see \rem{rem}). This fact will be used in the statements of some results below. Besides, in the definition of $\ZA_n^p(\Phi)$ instead of a $\Phi$-polynomial $f$ we can equivalently consider $f\in L^p$.
	Examples of wavelet type systems are considered in \rem{rem1} below.
	
	From this moment the constants occurring in the relations $\lesssim$ and $\sim$ may depend on the constants form \e{u43}, \e{u44} and the parameter  $1<p<\infty$. The main result of the paper is the following sharp estimate.
	\begin{theorem}\label{T1}
		If an orthonormal system $\Phi$ satisfies \e{u65}-\e{u44}, then for any $1<p<\infty$ it holds the relation
		\begin{equation}\label{a2}
			\ZA^p_{n}(\Phi)\lesssim \sqrt{\log (n+1)}, \quad n=1,2,\ldots.
		\end{equation}
	\end{theorem}
Note that in view of \e{1-2}, a likewise bound holds also for the sequences $\ZA^p_{n,\text{mon}}(\Phi)$ and $\ZA^p_{n,\text{sgn}}(\Phi)$. The following two corollaries show that for a wide class of wavelet type systems bound \e{a2} is sharp. Namely,
\begin{corollary}\label{C01}
	If an orthonormal system $\Phi$ satisfies \e{u65}-\e{u44} and is a basis in $L^2$,  then for any $1<p<\infty$ we have	
	\begin{equation}
	 \ZA^p_{n,\text{mon}}(\Phi) \sim  \ZA^p_{n}(\Phi) \sim\sqrt{\log (n+1)}.
	\end{equation}
\end{corollary}
\begin{corollary}\label{C02}
	If $\ZF$ is either Haar or Franklin system,  then for any $1<p<\infty$ we have	
	\begin{equation}\label{u70}
		\ZA^p_{n,\text{sng}}(\ZF)\sim \ZA^p_{n,\text{mon}}(\ZF) \sim  \ZA^p_{n}(\ZF) \sim\sqrt{\log (n+1)}.
	\end{equation}
\end{corollary}
The following bound is an interesting phenomenon of the wavelet type systems and immediately follows from \e{a2}. 
\begin{corollary}\label{C4}
	Let $\Phi=\{\phi_k\}$ be a wavelet type system and $G_k\subset \ZN$, $k=1,2,\ldots, n$ be a arbitrary sets of indexes. Then for any function $f\in L^p$, $1<p<\infty$, we have 
	\begin{equation}\label{u30}
		\left\|\max_{1\le m\le n}\left|\sum_{j\in G_m}\langle f,\phi_j\rangle \phi_j\right|\right\|_p\lesssim \sqrt{\log (n+1)}\cdot \|f\|_p.
	\end{equation}
\end{corollary}
In case of $p=2$ \trm{T1} implies the following results concerning the C or UC multipliers for non-overlapping $\Phi$-polynomials.
	\begin{corollary}\label{C1}
		If $\Phi=\{\phi_k\}$ is a wavelet type system, then the sequence  $\log n$ is a C-multiplier for any system of $L^2$-normalized \textit{non-overlapping} $\Phi$-polynomials
		\begin{equation}
			p_n(x)=\sum_{j\in G_n}c_j\phi_j(x),\quad n=1,2,\ldots,
		\end{equation}
		where $G_n\subset \ZN$ are finite and pairwise disjoint.
	\end{corollary}
	The following particular case of \cor{C1} is also new and interesting even for the classical Franklin system.
	\begin{corollary}\label{C3}
		The sequence $\log n$ is a C-multiplier for any rearrangement of a wavelet type system.
	\end{corollary}
	\begin{corollary}\label{C2}
		Let $\Phi=\{\phi_k\}$ be a wavelet type system and $\{p_n\}$ be a sequence of $L^2$-normalized non-overlapping $\Phi$-polyno\-mials. If $w(n)/\log (n+1)$ is increasing and
		\begin{equation}\label{omega}
			\sum_{n=1}^\infty \frac{1}{nw(n)}<\infty,
		\end{equation}
		then $w(n)$ is an UC-multiplier for $\{p_n\}$.
	\end{corollary}
	
	The only prior result in the context of Wail multipliers, concerning to non Haar wavelet type systems, is due to Gevorkyan \cite{Gev}, who proved 
	\begin{OldTheorem}[Gevorkyan, \cite{Gev}]
	The sequence $w(n)\nearrow$ is an UC-multiplier for the Franklin system if and only if it satisfies \e{omega}. 
	\end{OldTheorem}
\subsection{Remarks}
 \begin{remark}
 	In the case of the Franklin system the optimality of $\log n$ in \cor{C3}  as well as  condition \e{omega} in \cor{C2} both follows just from the direct combination of this result of Gevorkyan with  a result of Ul\cprime yanov-Poleshchuk \cite{Uly3, Pol}.
 \end{remark}

\begin{remark}
	\cor{C4} is a direct consequence of \trm{T1}. Corollaries \ref{C1}-\ref{C2} follow
from Theorem 1.1 in exactly the same way as Theorem C implies Corollaries 1.4 - 1.6
in \cite{Kar1}, therefore we have decided to skip the details.	The upper bounds in Corollaries \ref{C01} and \ref{C02} immediately follows from \trm{T1}. The lower bounds easily follow from certain results well-known in the theory of general orthogonal series and their proofs will be shortly provided immediately after the proof of \trm{T1}.
\end{remark}

\begin{remark}
	Similarly, as in the proofs of \trm{OT1} and \trm{OT2} the main argument in the proof of  \trm{T1} is a good-$\lambda$ inequality for the Haar system due to Chang-Wilson-Wolff  \cite{CWW}. The needed link between Haar and wavelet type expansions is provided by \lem{L4}. An application of the inequality of \cite{CWW} in the study of maximal functions of  Mikhlin-H\"{o}rmander multipliers was considered in \cite{GHS} by Grafakos-Honz\'{i}k-Seeger.  Namely, given multipliers $m_k$, $k=1,2,\ldots, N$, on $\ZR^n$ with uniform estimates, it was proved the optimal $\sqrt{ \log N}$ bound in $L^p$, $1<p<\infty$, for the maximal function $\max_{k} |\ZF^{-1}(m_k\hat f)|$.
\end{remark}
\begin{remark}\label{rem1}
	Examples of wavelet type systems are  the classical wavelets, orthonormal spline systems on $[0,1]$, where an interesting case is the Franklin system. All the results of the paper can be applied also to the periodic wavelet systems, since those can be split into a union of two wavelet type systems. With a slight change in the proofs one can also establish the results for the biorthogonal systems $(\phi_n,\psi_n)$, where both $\phi_n$ and $\psi_n$ satisfy the conditions \e{u43} and \e{u44}.

	Finally, observe that the Haar system doesn't satisfy \e{u44}. Nevertheless, the proof of estimate \e{u70} for the Haar system is straightforward and we will not need  \lem{L4}. Indeed, let $f=\sum_{j=1}^\infty b_jh_j\in L^p$ and $p_k\prec f$, $k=1,2,\ldots, n$. Instead of \e{u34} one needs to consider the function
	\begin{equation*}
		\ZP(x)=\sup_{1\le k\le n}S(p_k)(x)\le S\left(f\right)(x),
	\end{equation*}
	where $S$ is the Haar square function. So we have $\|\ZP\|_p\le c_p\|f\|_p$ (see \cite{KaSa}, chap. 3). Then repeating the argument of the proof of \trm{T1} given in \sect{S5}, we will get \e{u70} for the Haar system too.
\end{remark}

\begin{remark}\label{rem}
		Observe that the wavelet type systems share the following property of unconditional bases of $L^p$:  for any sequence $\lambda=\{\lambda_k\}$, $|\lambda_k|\le 1$, and a function $f\in L^p$, $1<p<\infty$, the series $T_\lambda(f)=\sum_k\lambda_k\langle f,\phi_k\rangle \phi_k$ converges in $L^p$ and $\left\|T_\lambda(f)\right\|_p\lesssim \|f\|_p$.
This property can be established by a well-known argument (see for example \cite{HeWe}), proving that $K_\lambda(x,t)=\sum_k \lambda_k\phi_k(x)\phi_k(t)$ is a Calder\'on-Zygmund kernel, namely
\begin{align}
	&\left|K_\lambda(x,t)\right|\lesssim \frac{1}{|x-t|},\label{y1}\\
	&|K_\lambda(x,t)-K_\lambda(x,t')|\lesssim \frac{|t-t'|^\beta }{|x-t|^{1+\beta}},\quad |x-t|>2|t-t'|,\label{y2}
\end{align}
where $0<\beta<\min\{\alpha,\delta\}$. Concerning to unconditional wavelet bases in $L^p$, $1<p<\infty$, we can also refer the readers to the papers \cite{Mey, Woj, Grip, Wol}.
\end{remark}
\begin{remark}
	Suppose $T_k$, $k=1,2,\ldots,n$, is a collection of Calder\'on-Zygmund operators with kernels, satisfying the standard conditions \e{y1} and \e{y2} uniformly.  It was proved in \cite{KaLa} that
	\begin{equation}\label{y4}
		\left\|\max_{1\le k\le n}|T_k(f)|\right\|_p\lesssim \log(n+1)\cdot \|f\|_p,\quad 1<p<\infty
	\end{equation}
(see \cite{KaLa}, Corollary 2.11). \cor{C4} above shows that for specific Calder\'on-Zygmund operators, namely
\begin{equation*}
	T_k(f)=\sum_{j\in G_k}\langle f,\phi_j\rangle \phi_j,\quad k=1,2,\ldots,n,
\end{equation*}
inequality \e{y4} holds with the bound $\sqrt{\log(n+1)}$.
\end{remark}
	
		\section{Notations and auxiliary estimates}
	Recall the definition of $L^2$-normalized Haar system $h_k(x)$, $k=1,2,\ldots$, on $[0,1]$. That is $h_1(x)=1$ and for $k\ge 2$ of the form \e{u52} we have
	\begin{align}
		h_k(x)=\left\{\begin{array}{rrll}
			&2^{n/2} &\hbox{ if }& x\in \left[{\frac{j-1}{2^n}},\frac{2j-1}{2^{n+1}}\right),\\
			&-2^{n/2} &\hbox{ if }& x\in \left[\frac{2j-1}{2^{n+1}},\frac{j}{2^n}\right),\\
			&0&\hbox{ otherwise. }&
		\end{array}
		\right.
		\end{align}
Given $1<q<\infty$ define the maximal function
\begin{equation*}
	\MM_q (f)(x)=\sup_{I:\,I\supset x}\left(\frac{1}{|I|}\int_I|f|^q\right)^{1/q},\quad f\in L^1(\ZT),
\end{equation*}
where $\sup$ is taken over all the open intervals $I\subset [0,1)$ containing the point $x$. The case of $q=1$, which is the classical Hardy-Littlewood maximal function, will be simply denoted  by $\MM$.
Recall the Fefferman-Stein vector valued maximal inequality (see \cite{Ste}, chap. 2)
\begin{equation}\label{u41}
	\left\|\left(\sum_{m=1}^\infty  |\MM_q\left(g_k\right)|^2\right)^{1/2}\right\|_p\lesssim \left\|\left(\sum_{m=1}^\infty  |g_k|^2\right)^{1/2}\right\|_p,
\end{equation}
where $1<p<\infty$, $1\le  q<\min\{2,p\}$. In this section we will establish standard estimates to be used in the proof of the main lemma.
	
	\begin{lemma}\label{L8}
		If $\{\phi_n\}$ is a wavelet type system, then for any coefficients $a_k$ there hold the bounds
		\begin{align}
			&\sum_{k=2^{m-1}+1}^{2^m}|a_k \phi_k(x)|\lesssim \MM\left(\sum_{k=2^{m-1}+1}^{2^m}a_k  h_k\right)(x),\label{u59}\\
			&\MM\left(\sum_{k=2^{m-1}+1}^{2^m}a_k  h_k\right)(x)\lesssim \MM\left(\sum_{k=2^{m-1}+1}^{2^m}|a_k \phi_k|\right)(x).\label{u45}
		\end{align}
	\end{lemma}
	\begin{proof}
		Using \e{u43} and a standard argument, we have 
		\begin{align}
			\sum_{k=2^{m-1}+1}^{2^{m}}|a_k\phi_k(x)|&\lesssim 2^{m/2}\sum_{k=2^{m-1}+1}^{2^{m}}|a_k|\xi\big(2^m(x-t_k)\big)\\
			&	\lesssim 2^{m}\int_0^1\sum_{k=2^{m-1}+1}^{2^{m}}|a_k||h_k(t)|\xi\big(2^m|x-t|\big)dt\\
			&\lesssim \MM\left(\sum_{k=2^{m-1}+1}^{2^{m}}a_kh_k\right)(x)
		\end{align}
		that gives \e{u59}. To prove \e{u45} we use \e{u44} and the normality condition $\|\phi_k\|_2=1$. So one can fix a number $a>0$, depending on the constant in \e{u43}, such that 
		\begin{equation}\label{u53}
			\int_{t_k-a/2^m}^{t_k+a/2^m}|\phi_k(t)|^2dx\ge \frac{1}{2},\quad 2^{m-1}<k<2^m.
		\end{equation}
		Given $\varepsilon >0$ consider the set
		\begin{equation*}
			E_k=\{x\in (t_k-a/2^m,t_k+a/2^m):\, |\phi_k(x)|>\varepsilon 2^{m/2}\}.
		\end{equation*}
		From \e{u53} and \e{u43} we obtain
		\begin{equation*}
			\frac{1}{2}\le c^22^m|E_k|+\varepsilon^22^m(2a/2^m-|E_k|),
		\end{equation*}
		that implies $|E_k|\gtrsim 1/n$ for a small enough $\varepsilon$ independent on $k$. Thus one can easily get
		
\begin{align*}
	\MM\left(\sum_{k=2^{m-1}+1}^{2^m}a_k  h_k\right)(x)&\lesssim\MM\left(\sum_{k=2^{m-1}+1}^{2^m}2^{m/2}|a_k| \ZI_{E_k}\right)(x)\\
	&\lesssim \MM\left(\sum_{k=2^{m-1}+1}^{2^m}|a_k|  |\phi_k|\right)(x)
\end{align*}
and so \e{u45}.
	\end{proof}
The following lemma is a version of Littlewood-Paley inequality.
\begin{lemma}\label{L11}
	Let $\{\phi_n\}$ be a wavelet type system and $1<p<\infty$. Then for any function $f\in L^p$ it holds the inequality
	\begin{equation}\label{u58}
		\left\|\left(\sum_{m=1}^\infty  \left(\sum_{k=2^{m-1}+1}^{2^{m}}|\langle f,\phi_k\rangle \phi_k|\right)^2\right)^{1/2}\right\|_p\lesssim \|f\|_p.
	\end{equation}
\end{lemma}
	\begin{proof} Denote $a_k=\langle f,\phi_k\rangle $ and let $r_k(x)$ be the Rademacher system. According to Remark \ref{rem} and Khintchin's inequality we have
\begin{equation}\label{u60}
\left\|\left(\sum_{k=1}^\infty  a_k^2 \phi_k^2\right)^{1/2}\right\|_p^p\sim \int_0^1\int_0^1\left|\sum_kr_k(t)a_k\phi_k(x)\right|^pdxdt\lesssim \|f\|_p^p.
\end{equation}
which is the classical version of the Littlewood-Paley inequality. 
Then, by \e{u45} in particular we have $|h_k(t)|\lesssim \MM(\phi_k)(t)$. Thus, from \e{u59}, \e{u41} with $q=1$ and \e{u60}, we obtain
		\begin{align}
			\left\|\left[\sum_{m=1}^\infty  \left(\sum_{k=2^{m-1}+1}^{2^{m}}|a_k\phi_k|\right)^2\right]^{1/2}\right\|_p&\lesssim \left\|\left[\sum_{m=1}^\infty  \MM\left(\sum_{k=2^{m-1}+1}^{2^{m}}a_kh_k\right)^2\right]^{1/2}\right\|_p\\
			&\lesssim \left\|\left(\sum_{k=2}^\infty a_k^2h_k^2\right)^{1/2}\right\|_p\lesssim \left\|\left[\sum_{k=2}^\infty (\MM(a_k\phi_k))^2\right]^{1/2}\right\|_p\\
			&\lesssim \left\|\left(\sum_{k=2}^\infty a_k^2\phi_k^2\right)^{1/2}\right\|_p\lesssim \|f\|_p.
		\end{align}
	\end{proof}

	\section{The main lemma}	\label{S4}
	Let $\Phi=\{\phi_k(x):\,k=1,2,\ldots\}$ be our wavelet type system, i.e. it  satisfies \e{u65}, \e{u43} and \e{u44}. For $f\in L^1(0,1)$, we consider the Fourier partial sums  
	\begin{align}
		&\Phi_n(f)(x)=\sum_{j=1}^{2^{n}}\langle f,\phi_j\rangle \phi_j(x),\quad n=0,1,2,\ldots,\\
		&\Delta \Phi_n(f)(x)=\sum_{j=2^{n-1}+1}^{2^{n}}\langle f,\phi_j\rangle \phi_j(x),\quad n\ge 1.
	\end{align}
	The Haar system analogues of such sums will be denoted by
	\begin{equation}\label{u32}
		H_{n}(f)(x)=\sum_{j=1}^{2^{n}}\langle f,h_{j}\rangle h_{j}(x),\quad \Delta H_n(f)(x)=\sum_{j=2^{n-1}+1}^{2^{n}}\langle f,h_j\rangle h_j(x).
	\end{equation}
	We denote by $I_{n}(x)$ the dyadic interval of the form $[(j-1)/2^n,j/2^n)$,
	containing a point $x\in[0,1)$. For the Haar partial sum \e{u32} we have
	\begin{equation}
		H_{n}(f)(x)=\frac{1}{|I_{n}(x)|}\int_{I_{n}(x)}f,\quad x\in [0,1)
	\end{equation}
(see \cite{KaSa}, chap. 3).  The following standard bound
\begin{align}
	\sum_{2^{n-1}<k\le 2^{n}}\left|\phi_k(x)\phi_k(t)\right|\lesssim 2^{n}\xi\big(2^{n}(x-t)\big)\label{y3}
\end{align}
one can find for example in \cite{HeWe} (see chap. 5, Lemma 3.12). 
	\begin{lemma}\label{L9}
		For any interval $I=[a,b)\subset [0,1)$ we have
		\begin{align}
			&|\Delta \Phi_m (\ZI_I)(x)|\lesssim  \frac{1}{(1+2^m|x-a|)^\delta}+\frac{1}{(1+2^m|x-b|)^\delta},\quad x\in [0,1),\label{u66}\\
			&|\Delta \Phi_m (\ZI_I)(x)|\lesssim  2^m|I|\xi(2^m(x-a)),\quad x\in [0,1]\setminus 2I.\label{u67}
		\end{align}
	\end{lemma}
	\begin{proof}
		To prove \e{u66} first consider the case $x\in [0,1)\setminus I$ and suppose that $0\le x<a$. Using \e{y3}, we obtain
		\begin{align}
			|\Delta \Phi_m (\ZI_I)(x)|&\lesssim \int_a^b\sum_{2^{m-1}<k\le 2^m}|\phi_k(x)\phi_k(t)|dt\lesssim \int_a^b\frac{2^m}{(1+2^m|x-t|)^{1+\delta}}dt\\
			&\lesssim \sum_{k\ge 2^m|x-a|}\frac{1}{k^{1+\delta }}\lesssim \frac{1}{(1+2^m|x-a|)^{\delta}}.
		\end{align}
		In the case $x\ge b$ we will have the bound $\lesssim \frac{1}{(1+2^m|x-b|)^{\delta}}$. If $x\in I$, then by \e{u65} we can write
		\begin{equation*}
			\Delta\Phi_m (\ZI_I)(x)=\Delta \Phi_m (\ZI_{[0,1)\setminus I})(x),
		\end{equation*}
		and \e{u66} can be obtained similarly. Then, under the condition $x\in [0,1]\setminus 2I$ we have 
		\begin{align*}
			|\Delta \Phi_m (\ZI_I)(x)|&\lesssim \int_I\frac{2^m}{(1+2^m|x-t|)^{1+\delta}}dt\\
			&\le |I|\cdot \max_{t\in I}\frac{2^m}{(1+2^m|x-t|)^{1+\delta}}\lesssim 2^m|I|\xi(2^m(x-a))
		\end{align*}
	that gives us \e{u67}.
	\end{proof}
For a function $f(t)$ defined on a set $E$  we denote 
\begin{equation*}
	\OSC_E(f)=\sup_{x,t\in E}|f(x)-f(t)|.
\end{equation*}
	\begin{lemma}\label{L1}
		If $g\in L^1(0,1)$, then for any integers $n= m\ge 1$ it holds the inequality
		\begin{equation}\label{x1}
			|\Delta H_{n}(\Delta \Phi_m(g))(x)|\lesssim 2^{\alpha (m-n)}\cdot \MM\left(\sum_{k=2^{m-1}+1}^{2^m}|\langle g,\phi_k\rangle||\phi_k|\right)(x),\quad x\in [0,1).
		\end{equation}
	\end{lemma}
	\begin{proof}
The case of $n=m$ is trivial. So we suppose $n>m$ and denote $a_k=\langle g,\phi_k\rangle$.	First using \e{u44} then \e{u45}, we obtain 
		\begin{align}
			|\Delta H_{n}(\Delta \Phi_m(g))(x)|&=\left|\frac{1}{|I_{n-1}(x)|}\int_{I_{n-1}(x)}\Delta \Phi_m(g)-\frac{1}{|I_{n}(x)|}\int_{I_{n}(x)}\Delta \Phi_m(g)\right|\\
			&\le \sum_{k=2^{m-1}+1}^{2^m} |a_k|\OSC_{I_{n-1}(x)}(\phi_k)\\
			&\lesssim 2^{m/2}\cdot 2^{\alpha (m-n)}\sum_{k=2^{m-1}+1}^{2^m} |a_k|\xi(2^m(x-t_k))\\
			&\lesssim 2^{\alpha (m-n)}\int_0^1\sum_{k=2^{m-1}+1}^{2^m} |a_k||h_k(t)|\cdot 2^m\xi(2^m(x-t))dt\\
			&\lesssim 2^{\alpha (m-n)}\cdot \MM\left(\sum_{k=2^{m-1}+1}^{2^m}a_k  h_k\right)(x)\\
			&\lesssim2^{\alpha (m-n)}\cdot \MM\left(\sum_{k=2^{m-1}+1}^{2^m}|a_k||\phi_k|\right)(x),
		\end{align}
	and lemma is proved.
	\end{proof}
	
	\begin{lemma}\label{L2}
		If $f\in L^1(0,1)$ and $m\ge n\ge 1$, then for any $q>1$ with $q'=q/(q-1)>\delta^{-1}$ we have
		\begin{equation}\label{x2}
			|H_{n}(\Delta \Phi_m(f))(x)|\lesssim 2^{(n-m)/q'}\MM_q (\Delta\Phi_m(f))(x),\quad x\in [0,1).
		\end{equation}
	\end{lemma}
	\begin{proof}
		Let $x\in [0,1)$. By the orthogonality of $\{\phi_n\}$ we have $\Delta \Phi_m(f)=\Delta \Phi_m(\Delta \Phi_m(f))$. Let $J_0=I_n(x)$, $J_k=2^kI_n(x)\cap [0,1]$, $k=1,2,\ldots $.
		Thus, using a duality argument, we can write
		\begin{align}
			\left|H_{n}(\Delta \Phi_m(f))(x)\right|&=\frac{1}{|I_{n}(x)|}\left|\int_{I_{n}(x)}\Delta \Phi_m(f)(u)du\right|\\
			&=\frac{1}{|I_{n}(x)|}\left|\int_{I_{n}(x)}\Delta \Phi_m(\Delta \Phi_m(f))(u)du\right|\\
			&=2^n\left|\int_0^1\Delta \Phi_m(\ZI_{I_n(x)})(u)\Delta \Phi_m(f)(u)du\right|\\
			&\le 2^n\left|\int_{J_0}\Delta \Phi_m(\ZI_{I_n(x)})(u)\Delta \Phi_m(f)(u)du\right|\\
			&\quad+2^n\sum_{k\ge 1}\left|\int_{J_k\setminus J_{k-1}}\Delta \Phi_m(\ZI_{I_n(x)})(u)\Delta \Phi_m(f)(u)du\right|.\label{u55}
				\end{align}
	Thus using \e{u66} for the first integral, we obtain
	\begin{align}
		2^n&\left|\int_{J_0}\Delta \Phi_m(\ZI_{I_n(x)})(u)\Delta \Phi_m(f)(u)du\right|\\
		&\qquad \lesssim {2^n}\left(\int_{J_0}|\Delta\Phi_m(f)(u)|^qdu\right)^{1/q}\cdot \left(\int_\ZR \frac{dt}{(1+2^m|t|)^{q'\delta}}\right)^{1/q'}\\
		&\qquad\lesssim 2^{(n-m)/q'}\MM_q (\Delta\Phi_m(f))(x).\label{u56}
	\end{align}
 From \e{u67} it follows that
\begin{equation*}
	|\Delta \Phi_m(\ZI_{I_n(x)})(u)|\lesssim2^{m-n}\left(\frac{2^{n-m}}{2^{k}}\right)^{1+\delta}=\frac{2^{\delta (n-m)}}{2^{k(1+\delta)}},\quad u\in J_k\setminus J_{k-1},
\end{equation*}
and so
	\begin{align}
		2^n&\left|\int_{J_k\setminus J_{k-1}}\Delta \Phi_m(\ZI_{I_n(x)})(u)\Delta \Phi_m(f)(u)du\right|\\
			&\qquad\qquad \lesssim2^n\frac{2^{\delta (n-m)}}{2^{k(1+\delta)}}\left|\int_{J_k}|\Delta \Phi_m(f)(u)|du\right|\\
			&\qquad\qquad \lesssim2^n\frac{2^{\delta (n-m)}}{2^{k(1+\delta)}}|J_k|\MM(\Delta \Phi_m(f))(x)\\
			&\qquad\qquad \le \frac{ 2^{\delta (n-m)}}{2^{k\delta}}\MM(\Delta \Phi_m(f))(x).\label{u57}
		\end{align}
		Hence, combining \e{u55}, \e{u56}  and \e{u57}, we get \e{x2}.
	\end{proof}
	Next, we will need a well-known discrete convolution inequality
	\begin{equation}\label{x4}
		\left(\sum_{n\in \ZZ}\left(\sum_{k\in \ZZ}|a_kb_{n-k}|\right)^2\right)^{1/2}\le \left(\sum_{n\in \ZZ} a_k^2\right)^{1/2}\cdot\left( \sum_{n\in \ZZ} |b_k|\right).
	\end{equation}
Also, recall the definition of the Haar square function 
\begin{equation*}
	S(f)(x)=\left(\sum_{k=1}^\infty |\langle f,h_k\rangle|^2h_k^2(x)\right)^{1/2}.
\end{equation*}
	\begin{lemma}[main]\label{L4}
		If $f\in L^p$, $1<p<\infty$, then
		\begin{equation}\label{x10}
			\left\|\sup_{|\lambda_k|\le 1}S\left(\sum_k \lambda_ka_k\phi_k\right)\right\|_p\lesssim \|f\|_p,
		\end{equation}
		where the $\sup$ is taken over all the sequences $\lambda=\{\lambda_k\}$ with $|\lambda_k|\le 1$.
	\end{lemma}
	\begin{proof}
		Clearly, we can suppose that $a_1=0$ and $0<\delta<1$. For a given sequence $\lambda=\{\lambda_k:\, |\lambda_k|\le 1\}$, we denote
		\begin{align}
			f_\lambda=\sum_{k\ge 1} \lambda_ka_k\phi_k=\sum_{k\ge 2} \lambda_ka_k\phi_k.
		\end{align}
		Chose a number $q=\min\{(p+1)/2,(1-\delta/2)^{-1}, 3/2\}$, which is depended only on $p$ and $\delta$. Combining \lem{L1},  \lem{L2}, and inequality \e{x4}, we can write 
		\begin{align}
			\sum_{n=1}^\infty \left|\Delta H_n\left(\sum_{k\ge 1} \lambda_ka_k\phi_k\right)(x)\right|^2&\le \sum_{n=1}^\infty\left[\sum_{m=1}^{\infty} 	|\Delta H_{n}(\Delta \Phi_m(f_\lambda))(x)|\right]^2\\
			&\lesssim \sum_{n=1}^\infty \left[\sum_{m=n+1}^\infty 2^{(n-m)/q'} \MM_q(\Delta \Phi_m(f_\lambda))(x)\right]^2\\
			&\quad +\sum_{n=1}^\infty \left[\sum_{m=1}^n2^{\alpha(m-n)}\big(\MM(\Delta \Phi_m(f_\lambda))(x)\big)\right]^2\\
			&\lesssim \sum_{n=1}^\infty \left[\sum_{m=n+1}^\infty 2^{(n-m)/q'} \MM_q\left(\sum_{k=2^{m-1}+1}^{2^{m}}|a_k\phi_k|\right)(x)\right]^2\\
			&\quad  +\sum_{n=1}^\infty  \left[\sum_{m=1}^n2^{\alpha(m-n)}\MM\left(\sum_{k=2^{m-1}+1}^{2^{m}}|a_k\phi_k|\right)(x)\right]^2\\
			&\lesssim \sum_{m=1}^\infty  \left[\MM_q\left(\sum_{k=2^{m-1}+1}^{2^{m}}|a_k\phi_k|\right)(x)\right]^2.\label{x7}
		\end{align}
		Thus we get the pointwise estimate
		\begin{equation*}
			\sup_{|\lambda_k|\le 1}S\left(\sum_{k\ge 1}\lambda_ka_k\phi_k\right)(x)\lesssim  \left[\sum_{m=1}^\infty  \left(\MM_q\left(\sum_{k=2^{m-1}+1}^{2^{m}}|a_k\phi_k|\right)(x)\right)^2\right]^{1/2}.
		\end{equation*}
		Then, using \e{u41} and \e{u58}, we obtain
		\begin{equation*}
			\left\|\sup_{|\lambda_k|\le 1}S\left(\sum_{k\ge 1}\lambda_ka_k\phi_k\right)\right\|_p\lesssim \left\|\left[\sum_{m=1}^\infty  \left(\sum_{k=2^{m-1}+1}^{2^{m}}|a_k\phi_k|\right)^2\right]^{1/2}\right\|_p\lesssim \|f\|_p.
		\end{equation*}
		Lemma is proved.
	\end{proof}

	\section{Proof of \trm{T1} and Corollaries \ref{C01} and \ref{C02}}\label{S5}
	A key argument in the proof of \trm{T1} is the following good-$\lambda$ inequality due to Chang-Wilson-Wolff  (see \cite{CWW}, Corollary 3.1):
	\begin{multline}\label{CWW}
		|\{x\in[0,1):\,\MM^d(f)(x)>\lambda,\, S f(x)<\varepsilon\lambda\}|\\\lesssim\exp\left(-\frac{c}{\varepsilon^2}\right)|\{\MM^d(f)(x)>\lambda/2\}|,\,\lambda>0,\,0<\varepsilon <1,
	\end{multline}
	where $\MM^d$ denotes the dyadic maximal function 
	\begin{equation*}
		\MM^d(f)(x)=\sup_{n\ge 1}\frac{1}{|I_{n}(x)|}\int_{I_{n}(x)}|f|.
	\end{equation*}
	Let $f\in L^p(0,1)$, $1<p<\infty$ and $a_k=\langle f,\phi_k\rangle$. Suppose the functions $p_k\in L^p(\ZT)$, $k=1,2,\ldots,n$ satisfy $p_k\prec f$ with respect to our wavelet type system $\Phi$. It is clear that
	\begin{equation}\label{u34}
		\ZP(x)=\sup_{1\le k\le n}S(p_k)(x)\le \sup_{|\lambda_k|\le 1}S\left(\sum_{k=0}^\infty\lambda_ka_k\phi_k\right)(x).
	\end{equation}
	Thus, according to \lem{L4}, we have
	\begin{equation}\label{u25}
		\|\ZP\|_p\lesssim \|f\|_p.
	\end{equation}
	On the other hand, $|g(x)|\le \MM^d g(x)$ a.e. for any function $g\in L^1$, as well as $S (p_k)(x)\le \ZP(x)$, $k=1,2,\ldots,n$. Thus, applying inequality \e{CWW} with $\varepsilon_n=(c/\ln n)^{1/2}$, we obtain 
	\begin{align}\label{b2}
		|\{|p_k(x)|>&\lambda,\, \ZP(x)\le\varepsilon_n\lambda\}|\\
		&\lesssim\exp\left(-\frac{c}{\varepsilon_n^2}\right)|\{\MM^dp_k(x)>\lambda/2\}|.
	\end{align}
	For $p^*(x)=\max_{1\le k\le n}|p_k(x)|$ we obviously have
	\begin{align}
		\{ p^*(x)>\lambda\}&\subset \{p^*(x)>\lambda,\, \ZP(x)\le \varepsilon_n\lambda\}\\
		&\cup  \{\ZP(x)> \varepsilon_n\lambda\}=A(\lambda)\cup B(\lambda),
	\end{align}
	and thus
	\begin{equation}
		\|p^*\|_p^p\le p\int_0^\infty\lambda^{p-1} |A(\lambda)|d\lambda+p\int_0^\infty\lambda^{p-1} |B(\lambda)|d\lambda.
	\end{equation}
	From \e{b2} it follows that
	\begin{align}
		\int_0^\infty\lambda^{p-1}|A(\lambda)|d\lambda&\le \sum_{k=1}^n\int_0^\infty\lambda^{p-1} |\{ |p_k|>\lambda,\, \ZP\le \varepsilon_n\lambda\}|d\lambda\\
		&\le \exp\left(-\frac{c}{\varepsilon_n^2}\right)\sum_{k=1}^n\int_0^\infty\lambda^{p-1} |\{ \MM^d p_k>\lambda/2\}|d\lambda\\
		&\lesssim \frac{1}{n}\sum_{k=1}^n\|\MM^d p_k\|_p^p\\
		&\lesssim \frac{1}{n}\sum_{k=1}^n\| p_k\|_p^p\\
		&\le \|f\|_p^p.
	\end{align}
	Combining this and 
	\begin{align*}
		p\int_0^\infty\lambda^{p-1}|B(\lambda)|d\lambda&=\varepsilon_n^{-p}\|\ZP\|_p^p\lesssim (\log (n+1))^{p/2} \cdot \|f\|_p^p,
	\end{align*}
	we get
	\begin{equation}\label{b4}
		\|p^*\|_p=\left\|\max_{1\le m\le n}|p_m(x)|\right\|_p\lesssim \sqrt{\log (n+1)}\cdot \|f\|_p
	\end{equation} 
	and so \e{a2}. 
	
	\begin{proof}[Proof of \cor{C01}]
	In view of \e{1-2} it remains to show 
	\begin{equation}\label{u40}
		\ZA^p_{n,\text{mon}}(\Phi)\gtrsim \sqrt{\log (n+1)}.
	\end{equation}
 We remarked in \cite{Kar1} that from some results of Nikishin-Ulyanov \cite{NiUl} and Olevskii \cite{Ole}  it follows that $\ZA^2_{n,\text{mon}}(\Phi)\gtrsim \sqrt{\log (n+1)}$ for any complete orthonormal system $\Phi$. 	Thus \e{u40} holds in the case $p=2$. So there is a function $f\in L^2$ and a sequence of integer sets $G_1\subset G_2\subset \ldots \subset G_n$ such that 
	\begin{equation*}
		\left\|\max_{1\le m\le n}\left|\sum_{j\in G_m}\langle f,\phi_j\rangle \phi_j\right|\right\|_2\ge c \sqrt{\log (n+1)}\cdot \|f\|_2,
	\end{equation*}
	Thus for the operator
	\begin{equation*}
		U(f)=\max_{1\le m\le n}\left|\sum_{j\in G_m}\langle f,\phi_j\rangle\phi_j\right|
	\end{equation*}
	we have $\|U\|_{2\to 2}\ge c\sqrt{\log (n+1)}$. On the other, according to \trm{T1}, we have $\|U\|_{p\to p}\le c_p\sqrt{\log (n+1)}$ for all $1<p<\infty$. Combining these two estimates with the "full version" of Marcinkiewicz interpolation theorem (see \cite{Zyg}, Theorem 4.6) one can easily obtain $\|U\|_{p\to p}\gtrsim \sqrt{\log (n+1)}$ that completes the proof of \cor{C01}.
\end{proof}
\begin{proof}[Proof of \cor{C02}]
	The upper bounds for both systems  follow from \trm{T1}. The lower bound $\ZA^p_{n,\text{sgn}}(\Phi)\gtrsim \sqrt{\log (n+1)}$  for the Franklin system if $p=2$ follows from the remark (1) given in  the introduction.  For the Haar system the case of $p=2$ of the same bound is known from \cite{NiUl}. To show the general case $1<p<\infty$ we again use the argument of the Marcinkiewicz interpolation theorem.
\end{proof}

	\bibliographystyle{plain}
	

\end{document}